\documentclass{amsart}
\newtheorem{Theorem}{Theorem}[section]
\newtheorem{Lemma}[Theorem]{Lemma}
\newtheorem{Corollary}[Theorem]{Corollary}

\newtheorem{Remark}[Theorem]{Remark}

\def\depth{\operatorname{depth}}
\def\reg{\operatorname{reg}}

\def\sk{\smallskip\par}
\def\mm{{\mathfrak m}}

\def\FilM{{\mathbb M}}

\def\FilE{{\mathbb E}}
\def\FilF{{\mathbb F}}

\begin{document}
\title{Hilbert coefficients of   good $I$-Filtrations of modules}
\thanks{{\it 2010 Mathematics Subject Classification:}  Primary 13D40, 13A30 \\ {\it Key words and phrases:} good filtration, associated graded module, Hilbert coefficients, Castelnuovo-Mumford regularity.}

\maketitle
\begin{center}
\begin{center}
	LE XUAN DUNG\\
	Faculty of  Natural Sciences,  Hong Duc University\\
	No 565 Quang Trung, Dong Ve, Thanh Hoa, Vietnam\\
	E-mail: lexuandung@@hdu.edu.vn, lxdung27@@gmail.com (L. X. Dung)\\ [15pt]
\end{center}
{\it Dedicated to Professor Le Tuan Hoa on the occasion of his 65$^{th}$ birthday}
  \end{center}

\begin{abstract} Let $M$ be a finitely generated module of dimention $d$ over a  Noetherian local ring $(A,\mm)$ and $I$ an $\mm$-primary ideal. Let be  a pair of good $I$-filtrations $\FilF$ and $\FilF'$ of $M$. We show that the Hilbert coefficients $e_i(\FilF)$ are bounded below and above in terms of $i$, $e_0(\FilF'),...,e_i(\FilF')$, and reduction numbers of $\FilF$ and  $\FilF'$, for all $i \ge 1$.
\end{abstract}

\date{}
\section{Introduction} \sk
Let $A$ be a commutative Noetherian local ring with the maximal ideal $\mm$ and $M$ be a finitely generated $A$-module of dimention $d$. Let $I$ be an ideal of $A$; an $I$-filtration $\FilF$ of $M$ is a collection of submodules $F_n$ such that
$$ M = F_0 \supseteq F_1 \supseteq F_2 \supseteq \cdots \supseteq F_n \supseteq  \cdots $$
with the property that $IF_n \subseteq  F_{n+1}$ for all $n \ge 0$. In the present work we consider only  good $I$-filtrations of $M$: this means that $IF_n =  F_{n+1}$ for all sufficiently large $n$.

The Hilbert-Samuel function $H_{\FilF} (n) = \ell(M/F_{n+1})$ agrees with the Hilbert-Samuel polynomial $P_{\FilF}(n)$ for $n \gg 0$ and we may write
$$P_{\FilF} (n) = e_0(\FilF){n+d \choose d} - e_1(\FilF){n+d-1 \choose d-1} + \cdots + (-1)^d e_d(\FilF).$$
The numbers $e_0(\FilF), e_1(\FilF), ...,e_d(\FilF)$ are called the Hilbert coefficients of $\FilF$. 

The notation of Hilbert function is central in communication algebra and is becoming increasingly importan in algebraic geometry and in computational algebra. Let be a good $I$-filtration $\FilF$ of $M$, the Hilbert-Samuel function and the  Hilbert-Samuel polynomial of  $\FilF$  give a lot of information on $M$. Therefore,   it is of interest to examine properties of the Hilbert coefficients of $\FilF$, see  (\cite{DH,DH2,GO,HH,KM,L2,Na,No,Rh,RTV,ST1,ST2,Tr2}). For further applications, we need to consider another filtration related to $I$ of $M$. Given a pair of good $I$-filtrations $\FilF$ and $\FilF'$ of $M$, we want to compare  $\FilF$ with $\FilF'$. Atiyah-Macdnald (\cite[Propsition 11.4]{AM}) and Brun-Hezog  (\cite[Proposition 4.6.5]{BH}) showed that $e_0(\FilF)=e_0(\FilF')$. In some special cases, Rossi-Vall in \cite{RV} gave alower bounds and upper bounds on $e_1(\FilF)$ in terms of $e_0(\FilF')$, $e_1(\FilF')$, and  other invarians of $M$. How about the other coefficients? The main goal of this paper is to show that $|e_i(\FilF)|$ are bounded by a function depeding only $i$, $e_0(\FilF'),...,e_i(\FilF')$, and reduction numbers of  $\FilF$ and $\FilF'$, for all $i \ge 1$ (see Theorem \ref{maintheorem2}). These bounds are far from being sharp, but they have some interest because very little is known about relationships between $e_0(\FilF), ..., e_d(\FilF)$ and $e_0(\FilF'),...,e_d(\FilF')$.

Our paper is outlined as follows. In the next section, we collect notations and terminology used in the paper and start with a few preliminary results on bounding the length of local homology modules (see Lemma \ref{A4} and Lemma \ref{A5}). In Section 3,  we give new bounds on the Castelnuovo-Mumford regularity $\reg(G(\FilF))$ of $\FilF$ (see Theorem \ref{maintheorem1}) and show that the Hilbert coefficients $e_i(\FilF)$ are bounded below and above in terms of $i$, $e_0(\FilF'),...,e_i(\FilF')$,   and reduction numbers of  $\FilF$ and $\FilF'$, for all $i \ge 1$ (see Theorem \ref{maintheorem2}).

\section{Hilbert coefficients and local cohomomology modules} \label{CMHil}
In this section, we recall notations and terminology used in the paper, and
a number of auxiliary results. Generally, we will follow standard texts in this
research area (cf. \cite{Bour,BS,RV}).

Let $R= \oplus_{n\ge 0}R_n$ be a Noetherian standard graded ring over a local Artinian ring $(R_0,\mm_0)$ such that $R_0/\mm_0$ is an infinite field. Let $E$ be a finitely generated graded $R$-module of dimension $d$. We denote the Hilbert function $\ell_{R_0}(E_t)$ and the Hilbert polynomial of $E$ by $h_E(t)$ and $p_E(t)$, respectively. Writing $p_E(t)$ in the form:
$$p_E (t) = \sum_{i=0}^{d-1} (-1)^i e_i(E){t+d-1-i \choose d-1-i},$$
we call  the numbers $e_i(E)$  {\it Hilbert coefficients} of $E$. 

Let $H_{R^+}^i(E)$, for $i  \ge 0$, denote the $i$-th local cohomology module of  $E$ with respect to $R^+$.  The {\it Castelnuovo-Mumford regularity of $E$} is defined by
$$\reg (E):= \max\{i+j|\,	H^i_{R^+}(E)_j \ne 0, 0 \le i \le d\}$$
and the {\it Castelnuovo-Mumford regularity of $E$  at and above level $1$}  is defined by 
$$\reg^1 (E):= \max\{i+j|\,	H^i_{R^+}(E)_j \ne 0, 0 < i \le d\}.$$

Let $\Delta (E)$ denote the maximal generating degree of $E$. From \cite[Theorem 2]{Tr1}, Dung-Hoa in \cite{DH2} derived an explicit bound for $\reg^1(E)$ in terms of $e_i(E)$, $0 \le i \le d-1$ and $\Delta'(E) = \max\{\Delta(E),\ 0\}.$

\begin{Lemma}\label{A1} \rm{(\cite[Lemma 1.2]{DH2})}
Let $E$ be a finitely generated  graded $R$-module of dimension $d\ge 1$.  
Put
$$\xi_{d-1}(E) = \max\{e_0(E), |e_1(E)|,...,|e_{d-1}(E)|\}.$$
 Then we have
$$\reg^1(E) \le (\xi_{d-1} (E) + \Delta'(E) + 1)^{d!} -2.$$
\end{Lemma}

Our method in proving the main result is to pass to the associated grade  modules, so we shall recall this notation and some more definitions.

 Let ($A,\mm$) be a Noetherian local  ring with an infinite residue field $K:= A/\mm$ and $M$ a finitely generated $A$-module. (Although the assumption $K$ being  infinite is not essential, because we can tensor $A$ with $K(t)$.)   Given a proper ideal $I$. A chain of submodules
$$\FilF:\ M = F_0 \supseteq F_1 \supseteq F_2 \supseteq \cdots \supseteq F_n \supseteq  \cdots $$
 is called an {\it $I$-filtration} of $M$ if $IF_i \subseteq  F_{i+1}$ for all $i$, and  a {\it good $I$-filtration} if $IF_i =  F_{i+1}$ for all sufficiently large $i$. A module $M$ with a filtration is called a {\it filtered module} (see \cite[III 2.1]{Bour}). If $N$ is a submodule of $M$, then the sequence $\{F_n+N/N\}$ is a good $I$-filtration of $M/N$ and will be denoted by $\FilF/N$.

Throughout the paper we always assume that $I$ is an $\mm$-primary ideal and $\FilF$ is a good $I$-filtration. The {\it associated graded module}  to the filtration $\FilF$ is defined by
$$G(\FilF) = \bigoplus _{n\geq 0}F_n/F_{n+1} .$$
We also say that $G(\FilF)$ is the associated ring of the filtered module $M$. This is a finitely generated graded module over the standard graded ring $G:= G(I,A) := \oplus_{n\ge 0}I^n/I^{n+1}$ (see \cite[Proposition III 3.3]{Bour}). In  particular, when $\FilF$ is the $I$-adic filtration $\{I^nM\}$, $G(\FilF)$ is just the usual associated graded module $G(I,M)$. 

 We call
$H_{\FilF}(n) = \ell(M/F_{n + 1})$
the Hilbert-Samuel function of $M$ w.r.t  $\FilF$. This function agrees with  a polynomial - called the Hilbert-Samuel polynomial  and denoted by $P_{\FilF}(n)$ - for $n \gg 0$.  If we write
$$P_\FilF (t) = \sum_{i=0}^d (-1)^i e_i(\FilF){t+d-i \choose d-i},$$
then the integers $e_i(\FilF)$ are called {\it Hilbert coefficients} of $\FilF$ (see \cite[Section 1]{RV}). When $\FilF = \{I^nM\}$, $H_{\FilF}(n)$ and $P_{\FilF}(n)$ are usually denoted by $H_{I,M}(n)$ and $P_{I,M}(n)$, respectively, and $e_i(\FilF) = e_i(I,M)$. Note that $e_i(\FilF)=e_i(G(\FilF))$ for $0\le i \le d-1$. Then

\begin{Lemma}\label{A11} {\rm (\cite[Proposition 11.4]{AM} and \cite[Proposition 4.6.5]{BH})}
	Let $\FilF$ and $\FilF'$ be good $I$-filtrations of $M$. Then we have
	$$ e_0(G(\FilF))=e_0(\FilF)=e_0(\FilF'). $$
\end{Lemma}
 We call
$$r(\FilF) = \min \{r \geq 0 \mid   F_{n+1} = IF_n \ \ \text{for all} \ \ n \geq  r \}$$
the reduction number of $\FilF$ (w.r.t. $I$). 

 In the case of $I$-adic filtration, $r(\FilF) = 0$. Note that $r:=r(\FilF)$ is always finite, and $F_{r+j} = I^jF_r$ for all $j\ge 0$. This means $\{ F_n\}_{n\ge r}$ is of form of an $I$-adic filtration of $F_r$. In other words, $r$ is the largest generating degree of $G(\FilF)$ as a graded module over $G$. 

Denote the filtration $\FilF/H^0_\mm(M) = \overline{\FilF}$. Let
$$h^0(M) = \ell(H^0_\mm(M)).$$
The
relationship between $\reg(G(\FilF))$ and $\reg(G(\overline{\FilF}))$  is given by the following lemma.
 \begin{Lemma}\label{A2} {\rm (\cite[Lemma 1.9]{DH})}
$\reg(G(\FilF)) \leq  \max \{ \reg(G(\overline{\FilF}));\ r(\FilF)\} + h^0(M).$
\end{Lemma}

 From now on, we will often use the following notation
 \begin{equation*}
\xi_s(\FilF) = \max\{e_0(\FilF), |e_1(\FilF)|,...,|e_s(\FilF)|\},
 \end{equation*}
 where $0\le s \le d$.  We see that
   \begin{equation}\label{eq-111}
  \xi_0(\FilF) \le \xi_1(\FilF) \le ... \le \xi_d(\FilF)=\xi(\FilF).
 \end{equation}

 An element $x \in I$ is called $\FilF$-{\it superficial element} for $I$ if there exists a non-negative integer $c$ such that $(F_{n+1}:_M x) \cap F_c = F_n$ for every $n \geq c$ and we say that a sequence of elements $x_1, ..., x_t$ is an $\FilF$-{\it superficial sequence} for $I$ if, for $i = 1, 2, ..., t$, $ x_i$ is an $\FilF/(x_1,...,x_{i-1})M$-superficial sequence for $I$  (see \cite[Section 1.2]{RV}).  The notion of superficial element is a fundamental tool in our work and we know that superficial sequence  of order $1$ always exist if the  residue field is infinite (see \cite[Proposition 8.5.7]{HS}).
 
 Using the \cite[Proposition 1.2 and Proposition 2.3]{RV} we get
 
  \begin{Lemma}\label{A3}  Let $x_1,...,x_d$ be an $\FilF$-superficial sequence for $I$ and $\overline{M}=M/H^0_\mm(M)$.  
  	Set $M_i=M/(x_1,...,x_i)M$ and $\FilF_i= \FilF/(x_1,...,x_i)M$, where $M_0=M$, $\FilF_0=\FilF$, $0\le i \le d-1$.
  	Then we have
  	
 {\rm i)} $ \xi_j(\overline{\FilF})= \xi_j(\FilF)$ for all $j \le d-1$,
 
{\rm ii)} $\xi_j(\overline{\FilF}/x_1\overline{M}) = \xi_j(\FilF)$ for all $j \le d-1$,
 
{\rm iii)} $ \xi_j(\FilF_i)=\xi_j(\FilF)$ for all $j \le d-i-1$.
 \end{Lemma}
 
 \begin{proof}
 	i) By \cite[Proposition 2.3]{RV}, $e_i(\FilF)=e_i(\overline{\FilF})$, for all $0 \le i \le d-1$. Hence $ \xi_j(\overline{\FilF})= \xi_j(\FilF)$ for all $j \le d-1$.\\
 	ii) We have $\depth(\overline{M})>0$, by \cite[Proposition 1.2]{RV}, 
 	$$e_i(\overline{\FilF}/x_1\overline{M})=e_i(\overline{\FilF}), \text{ for all } 0 \le i \le d-1.$$
 	Therefor  $$\xi_j(\overline{\FilF}/x_1\overline{M}) = \xi_j(\overline{\FilF}),  \text{ for all } 0 \le j \le d-1.$$
 	By i), we get $\xi_j(\overline{\FilF}/x_1\overline{M}) = \xi_j(\FilF)$ for all $j \le d-1$.\\
 	iii) By \cite[Proposition 1.2]{RV}, $\dim(M_{i-1})=d-i+1$ and
 	$$e_k(\FilF_i)=e_k(\FilF_{i-1}/x_iM_{i-1})=e_k(\FilF_{i-1}), \text{ for all } 0 \le k \le d-i-1.$$
 	Hence $e_k(\FilF_{i})=e_k(\FilF)$ for all $ 0 \le k \le d-i-1, 0  \le i \le d-1$. Therefor  $ \xi_j(\FilF_i)=\xi_j(\FilF)$ for all $j \le d-i-1$. 
\end{proof}

We can improve the bounds in \cite[Lemma 1.10 and Lemma 1.11]{DH2}. In the following results, we can replace $\reg(G(\FilF))$ by the Hilbert coefficents of $\FilF$. \begin{Lemma}\label{A4}  Let  $\FilF$ a good $I$-filtration of $M$ and  $x_1, x_2, ..., x_d$ be an $\FilF$-superficial sequence for $I$. Set $M_i = M/(x_1,...,x_i)M$ and $\FilF_i = \FilF/(x_1,...,x_i)M$  where $M_0 = M$ and $\FilF_0 = \FilF$. Then we have 	
$$h^0(M_i) \le \sum_{k=0}^{i} \xi_{d-i+k}(\FilF)(\xi_{d-i-1+k}(\FilF)+r(\FilF)+1)^{(d-i+k).(d-i+k)!},$$ for all $0 \leq i \leq d-1.$
\end{Lemma}
\begin{proof}
i) As mentioned above $G(\overline{\FilF})$ is generated by elements of degrees at most $ r(\overline{\FilM}) \ge 0$. Therefore, by \cite[Lemma 1.8]{DH} and Lemma \ref{A1}, we have 
\begin{align*}\reg(G(\overline{\FilF}_i)) = \reg^1(G(\overline{\FilF}_i)) & \leq
	(\xi_{d-i-1}(\overline{\FilF}_i)+\Delta'(\overline{\FilF}_i)+1)^{(d-i)!}-2  \\ 
	& = (\xi_{d-i-1}(\overline{\FilF}_i)+\Delta(\overline{\FilF}_i)+1)^{(d-i)!}-2  \\ 
	& =(\xi_{d-i-1}(\overline{\FilF}_i)+r(\overline{\FilF}_i)+1)^{(d-i)!}-2 .
\end{align*}	
 From Lemma \ref{A3} i) and iii) we get  $\xi_{d-i-1}(\overline{\FilF}_i) = \xi_{d-i-1}(\FilF_i)= \xi_{d-i-1}(\FilF)$ and $r(\overline{\FilF}_i) \le r(\FilF)$, therefore 
 $$\reg(G(\overline{\FilF_i})) \le (\xi_{d-i-1}(\FilF)+r(\FilF)+1)^{(d-i)!}-2 =:m_i.$$
 
For $i = 0$, by Lemma \cite[Lemma 1.6]{DH2}, we have
\begin{align*}
h^0(M_0) & =  h^0(M)  \leq P_{\FilF}(m_0)  \leq\xi_d(\FilF) \sum^d_{j = 0}{d+m_0-j \choose d-j} \\ 
&   = \xi_d(\FilF) {m_0+d+1 \choose d} \leq \xi_d(\FilF)(m_0+2)^d =\xi_d(\FilF) (\xi_{d-1}(\FilF)+r(\FilF)+1)^{d.d!}.
\end{align*}
          
For $0 < i \leq d-1$, by \cite[Proposition 1.2]{RV}, we have $e_j(\FilF_i) = e_j(\FilF_{i-1})$ for all $0 \leq j \leq d-i-1$. Similarly, as in the proof of \cite[Lemma 1.10]{DH2} and Lemma \ref{A3} iii) we have
\begin{equation*}
|e_{d-i}(\FilF_i)| \leq  \xi_{d-i}(\FilF_{i-1}) + h^0(M_{i-1}) \le \xi_{d-i}(\FilF) + h^0(M_{i-1}). 
\end{equation*}
It implies that
$$\begin{array}{lll}
h^0(M_i)
                & \le \xi_{d-i}(\FilF){m_i+d-i+1 \choose d-i} -\xi_{d-i}(\FilF) + |e_{d-i}(\FilF_i)|  & \\
                & \leq \xi_{d-i}(\FilF)(m_i+2)^{d-i}  + h^0(M_{i-1})  & \\
                          &  \leq \xi_{d-i}(\FilF) (\xi_{d-i-1}(\FilF)+r(\FilF)+1)^{(d-i)(d-i)!}+ &\\
                  &     + \sum_{k=0}^{i-1} \xi_{d-i+1+k}(\FilF)(\xi_{d-i+k}(\FilF)+r(\FilF)+1)^{(d-i+1+k).(d-i+1+k)!} & \\ & 
                  \, \, \, \, \, \, \, \, \, \, \, \, \, \, \, \, \, \, \, \, \, \, \, \, \, \, \, 
                   \, \, \, \, \, \, \, \, \, \, \, \, \, \, \, \, \, \, \, \, \, \, \, \, \, \, \, 
                    \, \, \, \, \, \, \, \, \, \, \, \, \, \, \, \, \, \, \, \, \, \, \, \, \, \, \, 
                     \, \, \, \, \, \, \, \, \, \, \, \, \, \, \, \, \, \, \, \, \, \, \, \, \, \, \,  (\text{by induction hypothesis})\\
 & =  \sum_{k=0}^{i} \xi_{d-i+k}(\FilF)(\xi_{d-i+k-1}(\FilF)+r(\FilF)+1)^{(d-i+k).(d-i+k)!}. & 
\end{array}$$
\end{proof}

\begin{Lemma} \label{A5} Set $B = \ell(M/(x_1,x_2,...,x_d)M)$, where $x_1, x_2, ..., x_d$ be an $\FilF$-superficial sequence for $I$ and put $\xi_{-1}=0$.  We have	
$$B \le  \sum_{k=0}^{d}\xi_{k}(\FilF)(\xi_{k-1}(\FilF)+r(\FilF)+1)^{k.k!}.$$
\end{Lemma}
\begin{proof}
Take  the proof of the \cite[Lemma 1.11]{DH2}. We have
\begin{equation}\label{eq:B4}
 B \leq e_0(\FilF) + h^0(M_{d-1}).
\end{equation}
 By Lemma \ref{A4}, $h^0(M_{d-1}) \le \sum_{k=0}^{d-1} \xi_{1+k}(\FilF)(\xi_{k}(\FilF)+r(\FilF)+1)^{(1+k).(1+k)!}$ .  From this  estimation we immediately get 
\begin{align*}
B & \le e_0(\FilF) + \sum_{k=0}^{d-1}\xi_{1+k}(\FilF)(\xi_k(\FilF)+r(\FilF)+1)^{(1+k)(1+k)!} \\
&= \xi_0(\FilF) + \sum_{k=1}^{d}\xi_{k}(\FilF)(\xi_{k-1}(\FilF)+r(\FilF)+1)^{k.k!}\\
&=  \sum_{k=0}^{d}\xi_{k}(\FilF)(\xi_{k-1}(\FilF)+r(\FilF)+1)^{k.k!}.
\end{align*}
\end{proof}

\section{Main results}
Throughout this section, $\FilF$ and $\FilF'$ will be a pair of good $I$-filtrations  of a finitely generated module $M$ over a local ring $(A,\mm)$, where $I$ is an $\mm$-primary ideal. The aim of this section is to show that the Hilbert coefficients $e_i(\FilF)$ are bounded below and above in terms of $e_0(\FilF'),...,e_i(\FilF')$, $i$, $r(\FilF)$, and $r(\FilF')$, for all $i \ge 1$. 

In order to prove the main result of this paper, we need bound on the  Castelnouvo-Mumford regularity $\reg(G(\FilF))$ of $\FilF$ in terms of $d$, $e_0(\FilF'),...,e_d(\FilF')$,  $r(\FilF)$, and  $r(\FilF')$.

\begin{Lemma} \label{B1}{\rm (\cite[Proof of Theorem 1.5]{DH})}
Let $\dim M=d\geq 2$, $x$ be an  $\FilF$-superficial sequence for $I$. We have
$$\reg^1(G(\overline{\FilF})/x^*G(\overline{\FilF})) = \reg^1(G(\overline{\FilF}/x\overline{M})).$$
\end{Lemma}

\begin{Theorem} \label{maintheorem1}
Let $\FilF$ and $\FilF'$ be are good $I$-filtrations of $M$ with  $\dim(M) = d \geq 1$ 
$$\FilF:\ M = F_0 \supseteq F_1 \supseteq F_2 \supseteq \cdots \supseteq F_n \supseteq  \cdots $$
$$\FilF':\ M = F'_0 \supseteq F'_1 \supseteq F'_2 \supseteq \cdots \supseteq F'_n \supseteq  \cdots $$
    Then
    
{\rm i)}  $\reg(G(\FilF)) \le  (\xi(\FilF')+r(\FilF')+1)(\xi(\FilF')+r(\FilF)+1)-2$ if $d=1$,

{\rm ii)}  $\reg(G(\FilF)) \le  (\xi(\FilF')+r(\FilF')+1)^6(\xi(\FilF')+r(\FilF)+1)-3$ if $d=2$,

{\rm iii)} $\reg(G(\FilF)) \le   (\xi(\FilF')+r(\FilF')+1)^{d(d+1)!-d}(\xi(\FilF')+r(\FilF)+1)^{(d-1)!}-d$ if $d \ge 3.$
\end{Theorem}

\begin{proof}   Let $\xi  := \xi(\FilF')$, $r:=r(\FilF)$ and $r':=r(\FilF')$. We distinguish two cases \\
If $d = 1$, then $\overline{M}$ is a Cohen-Macaulay module. By \cite[Lemma 1.8]{DH}, \cite[Lemma 2.2]{L1}, Lemma \ref{A11}, $r(\overline{\FilF}) \le r$ and (\ref{eq-111}) 
$$\begin{array}{ll} \reg(G(\overline{\FilF}))  & \leq e_0(G(\overline{\FilF})) + r(\overline{\FilF})  -1  \leq  e_0(\FilF') + r - 1
                                        \leq \xi + r -1.
\end{array}$$
Hence, by Lemma \ref{A2} and applying  Lemma \ref{A4} to $\FilF'$, we then obtain
$$ \begin{array}{ll}
\reg(G(\FilF)) & \leq  \max \{ \reg(G(\overline{\FilF}));\ r\} + h^0(M) \\
&\leq \xi  + r -1 + \xi(\xi+r'+1)\\
& \leq \xi  + r -1 + \xi(\xi+r')+(\xi+r')\\
&= (\xi  + r) + (\xi+1)(\xi+r')-1\\
&\leq (\xi  + r+1) + (\xi+r+1)(\xi+r')-2\\
& \le  (\xi+r'+1)(\xi+r+1)-2.
\end{array}$$

If $d \ge 2$, let $x_1,x_2,...,x_d$ be an $\FilF$-superficial sequence and $\FilF'$-superficial sequence for  $I$. Put $\overline{\FilF} = \FilF/H^0_\mm(M)$ and  $\overline{\FilF'} = \FilF'/H^0_\mm(M)$. We have $\overline{\FilF}/x_1\overline{M}$ and $\overline{\FilF'}/x_1\overline{M}$ be are good $I$-filtrations of $\overline{M}/x_1\overline{M}$ and $\dim (\overline{M}/x_1\overline{M}) =d-1$. Let  $m \geq   \max\{\reg(G(\overline{\FilF}/x_1\overline{M})), r\}$, 
by Lemma \ref{B1}, we have
$$\reg^1(G(\overline{\FilF})/x^*_1G(\overline{\FilF})) = \reg^1(G( \overline{\FilF}/x_1\overline{M})) \le m.$$
Hence, by \cite[Theorem 2.7]{L1},
$$\reg^1(G(\overline{\FilF})) \le m + P_{G(\overline{\FilF})}(m).$$
Since \cite[Lemma 1.6]{DH} and \cite[Lemma 1.7 (i)]{DH}
$$\begin{array}{ll}
 P_{G(\overline{\FilF})}(m)& \le H_{I,\overline{M}/x_1\overline{M}}(m) \\
 &  \le {m+d-1 \choose d-1} \ell \left(\left(\overline{M}/x_1\overline{M}\right)/(x_2,...,x_n) \left(\overline{M}/x_1\overline{M}\right)\right) \le B {m+d-1 \choose d-1}.
 \end{array}$$
 Therefor, by Lemma \ref{A2}, we get
 \begin{equation}\label{eq-2}
 	\reg(G(\FilF)) \le m + h^0(M)+  B {m+d-1 \choose d-1}.
 \end{equation}

If $d = 2$. Let $m= (\xi+r'+1)(\xi+r+1)-2$. Since (i) of the theorem, $r(\overline{\FilF'}/x_1\overline{\FilF'})\le r'$,  $r(\overline{\FilF}/x_1\overline{\FilF})\le r$ and by Lemma \ref{A3} ii), we get
$$\begin{array}{ll}
\reg(G(\overline{\FilF}/x_1\overline{M}))&\le (\xi_1(\overline{\FilF'}/x_1\overline{M})+r(\overline{\FilF'}/x_1\overline{M})+1)(\xi_1(\overline{\FilF'}/x_1\overline{M})+r(\overline{\FilF}/x_1\overline{M})+1)-2.\\
&= (\xi_1(\FilF')+r(\overline{\FilF'}/x_1\overline{M})+1)(\xi_1(\FilF')+r(\overline{\FilF}/x_1\overline{M})+1)-2.\\
& \le (\xi+r'+1)(\xi+r+1)-2=m.
\end{array}$$
Hence, $\max\{\reg(G(\overline{\FilF}/x_1\overline{M})), r\} \le m$. 
From (\ref{eq-111}), (\ref{eq-2}),  and applying Lemma \ref{A4}, Lemma \ref{A5} to $\FilF'$, we get
{\small
$$\begin{array}{ll}
\reg(G(\FilF))&\le  m+h_0(M)+B(m+1)\\
&\le (\xi+r'+1)(\xi+r+1)-2+\xi(\xi+r'+1)^4+ \\ 
&+[\xi + \xi(\xi+r'+1) + \xi(\xi+r'+1)^4][(\xi+r'+1)(\xi+r+1)-1]\\
&\le (\xi+r'+1)(\xi+r+1)+\xi(\xi+r'+1)^3(\xi+r'+1)(\xi+r+1)+ \\ 
&+[\xi + \xi(\xi+r'+1)^2 + \xi(\xi+r'+1)^4](\xi+r'+1)(\xi+r+1) -3 \\
&\le[1+ \xi+\xi(\xi+r'+1)^2 + \xi(\xi+r'+1)^3  +  \xi(\xi+r'+1)^4] (\xi+r'+1)(\xi+r+1) -3 \\
& \le (\xi+r'+1)^5(\xi+r'+1)(\xi+r+1) -3 \\
& = (\xi+r'+1)^6(\xi+r+1) -3.
\end{array}$$}

If $d\ge 3$. 
By the induction hypothessis, $r(\overline{\FilF}/x_1\overline{M})\le r$,  $r(\overline{\FilF'}/x_1\overline{M})\le r'$ and by Lemma \ref{A3} ii), we have 
{\footnotesize
	$$\begin{array}{ll}
	\reg(G(\overline{\FilF}/x\overline{M}))& \le  (\xi_{d-1}(\overline{\FilF'}/x_1\overline{M})+r(\overline{\FilF'}/x_1\overline{M})+1)^{(d-1)d!-d+1}(\xi_{d-1}(\overline{\FilF'}/x_1\overline{M})+r(\overline{\FilF}/x_1\overline{M})+1)^{(d-2)!}-d+1\\
	& =  (\xi_{d-1}(\FilF')+r(\overline{\FilF'}/x_1\overline{M})+1)^{(d-1)d!-d+1}(\xi_{d-1}(\FilF')+r(\overline{\FilF}/x_1\overline{M})+1)^{(d-2)!}-d+1\\
	& \le (\xi+r'+1)^{(d-1)d!-d+1}(\xi+r+1)^{(d-2)!}-d+1.
	\end{array}$$}
We can take 
$$m= (\xi+r'+1)^{(d-1)d!-d+1}(\xi+r+1)^{(d-2)!}-d+1 \geq 2.$$
We see that 
$$1+m+{m+d-1 \choose d-1} \le  (m+1)^{d-1} \text{ for all } m \ge 2.$$
Therefore, by (\ref{eq-2}) 
 and applying Lemma \ref{A4}, Lemma \ref{A5}  to $\FilF'$, we get
{\footnotesize
	\begin{align}
		\reg(G(\FilF)) 		& \leq  m+ \xi_d(\FilF') (\xi_{d-1}(\FilF')+r'+1)^{d.d!} +\sum_{k=0}^{d}\xi_{k}(\FilF')(\xi_{k-1}(\FilF')+r'+1)^{k.k!}{m+d-1 \choose d-1} \nonumber  \\	
			& < \sum_{k=0}^{d}\xi_{k}(\FilF')(\xi_{k-1}(\FilF')+r'+1)^{k.k!}\left[ 1+m+{m+d-1 \choose d-1}\right] -d\nonumber  \\			
		& < (\xi_{d}(\FilF')+r' + 1)^{d.d!+1}(m+1 )^{d-1} -d  \nonumber \\
				& \leq (\xi_{d}(\FilF')+r' + 1)^{d.d!+1}\left[(\xi+r'+1)^{(d-1)d!-d+1}(\xi+r+1)^{(d-2)!}-d+2\right]^{d-1} -d  \nonumber \\
					& <  (\xi+r'+1)^{d.d!+1 +  [(d-1)d!-d+1](d-1)}(\xi+r+1)^{(d-1)!}-d. \nonumber
\end{align}}
Since $d \ge 3$, the following hold
{\footnotesize
$$	\begin{array}{ll}
		& d.d!+1 +  [(d-1)d!-d+1](d-1) - [d(d+1)!-d]  \\	
		 = & [d+(d-1)^2-d(d+1)]d!+1-(d-1)^2+d \\
		= & (1-2d)d!+3d-d^2 < 0 . \\	
\end{array}
$$
}
Hence $	\reg(G(\FilF)) 		 \leq   (\xi+r'+1)^{d(d+1)!-d}(\xi+r+1)^{(d-1)!}-d.$
\end{proof}

Now we are going to prove the main result of this paper.
\begin{Theorem} \label{maintheorem2}
	Let $\FilF$ and $\FilF'$ be good $I$-filtrations of $M$ with  $\dim(M) = d \geq 1$ 
	$$\FilF:\ M = F_0 \supseteq F_1 \supseteq F_2 \supseteq \cdots \supseteq F_n \supseteq  \cdots $$
	$$\FilF':\ M = F'_0 \supseteq F'_1 \supseteq F'_2 \supseteq \cdots \supseteq F'_n \supseteq  \cdots $$
	Then
	
	{\rm i)}  $|e_1(\FilF)| \le \xi_1(\FilF')(\xi_1(\FilF')+r(\FilF')+1)^{2}(\xi_1(\FilF') + r(\FilF)  +1)$;
	
	{\rm ii)} $|e_2(\FilF)| \le \xi_2(\FilF')(\xi_2(\FilF')+r(\FilF')+1)^{17}(\xi_2(\FilF') + r(\FilF)  +1)^2$;
	
	{\rm iii)} $|e_i(\FilF)| \le \xi_i(\FilF')(\xi_i(\FilF')+r(\FilF')+1)^{(i^3+i^2+i)i!-i^2+1}(\xi_i(\FilF') + r(\FilF)  +1)^{i!}$ if $i\ge 3$.
\end{Theorem}

\begin{proof}

	i)  By \cite[(8)]{DH} we have 
	\begin{equation} \label{eq-4} 
	\ell(M/F_{m+1}) = \sum_{i=0}^d (-1)^i e_i(\FilF){m+d-i \choose d-i}
	\end{equation}
	for any $m\ge\reg(G(\FilF))$. For short we write $\xi_i:=\xi_i(\FilF'), r:=r(\FilF)$, and $r':=r(\FilF')$.
	
	Assume that $d=1$. Putting $m:=(\xi_1+r'+1)(\xi_1+r+1)-1$, by Theorem \ref{maintheorem1} i) and (\ref{eq-4}), we have
		\begin{equation} \label{eq-44} 
		e_1(\FilF)=(m+1)e_0(\FilF)-\ell(M/F_{m+1})
			\end{equation}
	By (\ref{eq-44}) and Lemma \ref{A11}, this implies
			\small{	$$
		\begin{array}{ll}
	e_1 & \le (\xi_1+r'+1)(\xi_1+r+1)\xi_0
	 \le \xi_1(\xi_1+r'+1)^2(\xi_1+r+1).	
	\end{array}$$}
	By \cite[Lemma 1.7 i)]{DH}, Lemma \ref{A11} and Lemma \ref{A5}
				$$
	\begin{array}{ll}
-	e_1(\FilF) & \le B(m+1)-(m+1)e_0(\FilF) =(B-\xi_0)(m+1)          \\   
	& \le  \xi_1(\xi_1+r'+1)(\xi_1+r'+1)(\xi_1+r+1)  = \xi_1(\xi_1+r'+1)^2(\xi_1+r+1).	
	\end{array}$$
	Hence
	$$|e_1(\FilF)| \le \xi_1(\xi_1+r'+1)^{2}(\xi_1 + r  +1).$$
	
		Assume that $d\ge 2$. Let $x_1,...,x_d$ be  $\FilF$-superficial sequence and $\FilF'$-superficial sequence for $M$ and $I$. Put $\FilF_0=\FilF$, $\FilF'_0=\FilF'$, $N_0=M$ and $\overline{\FilF}_0 = \FilF_0/H^0_\mm(N_0)$, $\overline{\FilF'}_0 = \FilF_0'/H^0_\mm(N_0)$, $\overline{N}_0 = N_0/H^0_\mm((N_0)$.  We have $\FilF_i=\overline{\FilF}_{i-1}/x_i\overline{N}_{i-1}$,  $\FilF'_i=\overline{\FilF'}_{i-1}/x_i\overline{N}_{i-1}$ be are good $I$-filtrations of $N_i=\overline{N}_{i-1}/x_i\overline{N}_{i-1}$, and $\dim N_i = d-i $ for all $i$, $1 \leq i \leq d$. By \cite[Proposition 1.2  and  Prosition 2.3]{RV}, we get
		\begin{equation}\label{eq-5}
		e_i(\FilF)=e_i(\FilF_{d-i}) \text{  for all } i \le d-1 .
		\end{equation}
		By Theorem \ref{maintheorem1}, $\reg(G(\FilF))  \le m$, (\ref{eq-4}), \cite[Lemma 1.7 ii)]{DH} and \cite[Corollary 4.7.11 a)]{BH}, we have 
		\footnotesize{	\begin{eqnarray}
			|e_d(\FilF)|&=&\left| \ell(M/F_{m+1})-e_0(\FilF) {m+d \choose d}+...+ (-1)^de_{d-1}(\FilF)(m+1) \right| \nonumber \\	
			& \le & \max\left\{B{m+d \choose d}, e_0(\FilF) {m+d \choose d}\right\}+\sum_{i=1}^{d-1}| e_i(\FilF)|{m+d-i \choose d-i} \nonumber \\
			& \le & B(m+d)^d+ \sum_{i=1}^{d-1}| e_i(\FilF)|{m+d-i \choose d-i} \label{eq-51}
				\end{eqnarray}}

		 If $d=2$, by (\ref{eq-5}), $e_1(\FilF)=e_1(\FilF_1)$.	
		 Using the induction hypothessis, $r(\FilF_1)\le r$,  $r(\FilF'_1)\le r'$ and by Lemma \ref{A3} ii), we have \small{
		 	$$\begin{array}{ll}
		 	|e_1(\FilF)|&\le \xi_1(\FilF'_1)(\xi(\FilF'_1)+r(\FilF'_1)+1)^2 (\xi(\FilF'_1)+r(\FilF_1)+1)  \\
		 	& =  \xi_1(\xi_1+r'+1)^{2}(\xi_1 + r  +1).
		 	\end{array}$$}

	 	By Lemma \ref{A5}, Theorem \ref{maintheorem1} ii) and putting 
	 	$m=(\xi+r'+1)^6(\xi+r+1)-2$ into (\ref{eq-51}), we have

{\small
		$$ \begin{array}{ll}
		|e_2(\FilF)| &  \le B(m+2)^2+| e_1(\FilF)|(m+1)\\
		& \le B\left[ (\xi+r'+1)^6(\xi+r+1)\right]^2 + \xi(\xi+r'+1)^2(\xi+r+1) [(\xi+r'+1)^6(\xi+r+1)-1]\\
		& < [\xi + \xi(\xi+r'+1)+\xi(\xi+r'+1)^4]\left[ (\xi+r'+1)^6(\xi+r+1)\right]^2+\\
		& + \xi(\xi+r'+1)^2(\xi+r+1) (\xi+r'+1)^6(\xi+r+1)\\
		& \le \xi[1+(\xi+r'+1)^2+(\xi+r'+1)^4+1 ]\left[ (\xi+r'+1)^6(\xi+r+1)\right]^2\\
		&\le \xi(\xi+r'+1)^5(\xi+r'+1)^{12}(\xi+r+1)^2\\
		& \le \xi(\xi+r'+1)^{17}(\xi+r+1)^2.
		\end{array}$$}	
	iv)  Assume that $d\ge 3$. Using the induction hypothessis, $r(\FilF_i)\le r$,  $r(\FilF'_i)\le r'$,  for all $1 \le i \le d-1$ and by Lemma \ref{A3} ii), we have
{\small		 
	 \begin{eqnarray}
	 |e_1(\FilF)|& = & |e_1(\FilF_{d-1})| \le \xi_1(\FilF'_{d-1})(\xi_1(\FilF'_{d-1})+r(\FilF'_{d-1})+1)^{2}(\xi_1(\FilF'_{d-1}) + r(\FilF_{d-1})  +1) \nonumber \\
	 & \le &\xi_1 (\xi_1+r'+1)^{2}(\xi_1 + r  +1).  \label{eq-6}\\
	 |e_2(\FilF)|& = &|e_2(\FilF_{d-2})| \le \xi_2(\FilF'_{d-2})(\xi_2(\FilF'_{d-2})+r(\FilF'_{d-2})+1)^{17}(\xi_2(\FilF'_{d-2}) + r(\FilF_{d-2})  +1)^2  \nonumber\\
	 & \le & \xi_2(\xi_2+r'+1)^{17}(\xi_2 + r  +1)^2.  \label{eq-7}\\
	 |e_i(\FilF)|& = &|e_i(\FilF_{d-i})| \le \xi_i(\FilF'_{d-i})(\xi_i(\FilF'_{d-i})+r(\FilF'_{d-i})+1)^{(i^3+i^2+i)i!-i^2+1}(\xi_i(\FilF'_{d-i}) + r(\FilF_{d-i})  +1)^{i!} \nonumber\\
	 & \le & \xi_i(\xi_i+r'+1)^{(i^3+i^2+i)i!-i^2+1}(\xi_i + r  +1)^{i!} 	 \text{ if }  3\le i \le d-1.\label{eq-8}
	 \end{eqnarray}
	 }
 To prove the inequallity for $e_d(\FilF)$, we set 
	$$m=(\xi+r'+1)^{d(d+1)!-d}(\xi+r+1)^{(d-1)!}-d.$$
	By (\ref{eq-51}),  Theorem \ref{maintheorem1}, $\reg(G(\FilF))  \le m$ and (\ref{eq-4}), we have 
		\footnotesize{	\begin{eqnarray}
			|e_d(\FilF)|& \le & B(m+d)^d+| e_1(\FilF)|(m+d-1)^{d-1}+ \sum_{i=2}^{d-1}| e_i(\FilF)|(m+d-i)^{d-i} \nonumber \\
			& \le & B(m+d)^d+| e_1(\FilF)|(m+d)^{d-1}+ \sum_{i=2}^{d-1}| e_i(\FilF)|(m+d)^{d-i} \nonumber \\
		& = & \left(B+\dfrac{| e_1(\FilF)|}{m+d}+\dfrac{| e_2(\FilF)|}{(m+d)^2}+\sum_{i=3}^{d-1}\dfrac{ | e_i(\FilF)|}{(m+d)^i}\right)(m+d)^d.	\label{eq-9}	
		\end{eqnarray}}
		By (\ref{eq-6})-(\ref{eq-8}), we get	
	\begin{eqnarray}
	\dfrac{| e_1(\FilF)|}{m+d}& \le & \dfrac{\xi_1(\FilF')(\xi_1(\FilF')+r'+1)^{2}(\xi_1(\FilF') + r  +1)}{(\xi+r'+1)^{d(d+1)!-d}(\xi+r(\FilF)+1)^{(d-1)!} } \nonumber \\
	& \le & \dfrac{\xi(\xi+r'+1)^{2}(\xi + r  +1)}{(\xi+r'+1)^{d(d+1)!-d}(\xi+r+1)^{(d-1)!} }\le 
	  \dfrac{\xi}{2}.\label{eq-10} \\ 
	\dfrac{| e_2(\FilF)|}{(m+d)^2} &  \le & \dfrac{\xi_2(\FilF')(\xi_2(\FilF')+r'+1)^{17}(\xi_2(\FilF') + r  +1)^2}{\left[(\xi+r'+1)^{d(d+1)!-d}(\xi+r+1)^{(d-1)!}\right]^2} \nonumber  \\
	&  \le & \dfrac{\xi(\xi+r'+1)^{17}(\xi + r  +1)^2}{\left[(\xi+r'+1)^{d(d+1)!-d}(\xi+r+1)^{(d-1)!}\right]^2} \le \dfrac{\xi}{2^{2}}.\label{eq-11} \\ 
		\dfrac{| e_i(\FilF)|}{(m+d)^i}& \le &  \dfrac{\xi_i(\FilF')(\xi_i+r'+1)^{(i^3+i^2+i)i!-i^2+1}(\xi_i(\FilF') + r  +1)^{i!}}{\left[(\xi+r'+1)^{d(d+1)!-d}(\xi+r+1)^{(d-1)!}\right]^i} \nonumber \\
		&  \le & \dfrac{\xi(\xi+r'+1)^{(i^3+i^2+i)i!-i^2+1}(\xi + r  +1)^{i!}}{\left[(\xi+r'+1)^{d(d+1)!-d}(\xi+r+1)^{(d-1)!}\right]^i}   \le \dfrac{\xi}{2^{i}} \nonumber  \\
		 \text{ if } 3\le i \le d-1.	 \label{eq-12}
		\end{eqnarray}
From (\ref{eq-9})-(\ref{eq-12}) and Lemma \ref{A5}, we obtain	
										\small{
				$$ \begin{array}{ll}
				|e_d(\FilF)| & \le \left[B+\xi\left(\dfrac{1}{2}+...+\dfrac{1}{2^{d-1}} \right) \right](m+d)^d < (B+\xi) (m+d)^d\\
				& \le \xi\left[1+(\xi+r'+1)+(\xi+r'+1)^{2.2!}+...+(\xi+r'+1)^{d.d!}+1\right] \times \\
				 & \times  \left[(\xi+r'+1)^{d(d+1)!-d}(\xi+r+1)^{(d-1)!}\right]^d\\
				 & \le \xi	(\xi+r'+1)^{d.d!+1}(\xi+r'+1)^{d^2(d+1)!-d^2}(\xi+r+1)^{d!}	\\
				 & = \xi(\xi+r'+1)^{d^2(d+1)!-d^2+d.d!+1}(\xi+r+1)^{d!}  \\
				& = \xi(\xi+r'+1)^{(d^3+d^2+d)d!-d^2+1}(\xi + r  +1)^{d!}.
				\end{array}$$}
\end{proof}

We immediately obtain the following consequence 
\begin{Corollary}\label{co1}
		Let $\FilF$ be a good $I$-filtration of $M$ with  $\dim(M) = d \geq 1$. Then

	{\rm i)}  $|e_1(\FilF)| \le \xi_1(I,M)(\xi_1(I,M)+1)^{2}(\xi_1(I,M) + r(\FilF)  +1)$;
	
	{\rm ii)} $|e_2(\FilF)| \le \xi_2(I,M)(\xi_2(I,M)+1)^{17}(\xi_2(I,M) + r(\FilF)  +1)^2$;
	
	{\rm iii)} $|e_i(\FilF)| \le \xi_i(I,M)(\xi_i(I,M)+1)^{(i^3+i^2+i)i!-i^2+1}(\xi_i(I,M) + r(\FilF)  +1)^{i!}$ if $i\ge 3$.
\end{Corollary}
\begin{proof} The reduction number of the $I$-adic filtration $\{I^nM\}$ is $0$.  Therefore, applying Theorem \ref{maintheorem2} to $\FilF' = \{I^nM\}$, we then obtain.
\end{proof}

Let $x_1,...,x_d$ be an $\FilF$-superficial sequence for $I$ and $Q:=\left(x_1,...,x_d\right)$.  It is not difficult to prove that also the  $\FilF$ is a  good  $Q$-filtration of $M$. Rossi-Valla in \cite{RV} gave the following filtration
$$\FilE:\ M = F_0 \supseteq F_1 \supseteq QF_1 \supseteq Q^2F_1 \supseteq \cdots \supseteq Q^nF_1 \supseteq  \cdots.$$
\noindent This filtration is  a  good  $Q$-filtration of $M$. As in consequence of the  Theorem \ref{maintheorem2} we have a relationship between $\FilE$ and $\{Q^nM\}$ as follows:

\begin{Corollary} \label{co2} Let $x_1,...,x_d$ be an $\FilF$-superficial sequence for $I$ and $Q:=\left(x_1,...,x_d\right)$. Then

	{\rm i)}  $|e_1(\FilE)| \le \xi_1(Q,M)(\xi_1(Q,M)+1)^{2}(\xi_1(Q,M) + 2)$;
	
	{\rm ii)} $|e_2(\FilE)| \le \xi_2(Q,M)(\xi_2(Q,M)+1)^{17}(\xi_2(Q,M) + 2)^2$;
	
	{\rm iii)} $|e_i(\FilE)| \le \xi_i(Q,M)(\xi_i(Q,M)+1)^{(i^3+i^2+i)i!-i^2+1}(\xi_i(Q,M) + 2)^{i!}$ if $i\ge 3$.
\end{Corollary}
\begin{proof}
The reduction number of the good $Q$-filtration $\FilE$ is $1$ and the reduction number of  $Q$-adic filtration $\{Q^nM\}$ is $0$.  Therefore, applying Theorem \ref{maintheorem2} to $\FilF = \FilE$ and $\FilF' = \{Q^nM\}$, we then obtain.
\end{proof}
\begin{Remark}{\rm
Let $p$ be an integer such that $IM \subseteq \mm^pM$.	Rossi-Valla in \cite[Proposition 2.10 and Proposition 2.11]{RV} gave a sharp upper  bounds for $e_1(\FilF)$ in terms of $e_0(Q,M)$, $e_1(Q,M)$, and $p$ and a sharp lower  bounds for $e_1(\FilE)$ in terms of $e_0(Q,M)$, $e_1(Q,M)$ and other invarians of $M$, respectively. The bounds of Corollary \ref{co1} and Corollary \ref{co2}  are far from being sharp, but they 
show that the Hilbert coefficients $e_i(\FilF)$ and $e_i(\FilE)$ are bounded below and above in terms of $e_0(Q,M),...,e_i(Q,M)$, $i$, and  $r(\FilF)$ (only for $e_i(\FilF)$),  for all $i \ge 1$. 
}	
\end{Remark}
\vskip0.5cm


\end{document}